\documentclass[11pt,a4paper,twoside]{article}

\usepackage{amsfonts}
\usepackage{mathrsfs}
\usepackage{amsthm}
\usepackage{amsmath}%
\usepackage{amssymb}%
\usepackage{fancyhdr}%
\usepackage{ifthen}%
\usepackage{calc}%
\usepackage{lastpage}%
\usepackage{array}%
\usepackage{anysize}%
\usepackage{cases}

\makeatletter
\renewcommand\@biblabel[1]{${#1}.$}
\makeatother

\title{\bf Stochastic Analysis on Path Space over Time-Inhomogeneous Manifolds with Boundary}
\author{{\bf Li-Juan Cheng \footnote{Correspondence should be addressed to Li-Juan Cheng (E-mail: chenglj@mail.bnu.edu.cn)} }\\
{\scriptsize (School of Mathematical Sciences,  Beijing Normal
University, } \\{\scriptsize Laboratory of Mathematics and Complex
Systems,  Ministry of Education, } \\ {\scriptsize Beijing 100875,
The People's Republic of China)}\\
{\scriptsize E-mail: chenglj@mail.bnu.edu.cn(L.J. Cheng)}}
\date{}

\newtheorem{theorem}{Theorem}[section]
\newtheorem{corollary}[theorem]{Corollary}
\newtheorem{lemma}[theorem]{Lemma}
\newtheorem{proposition}[theorem]{Proposition}

\newtheorem{remark}[theorem]{Remark}
\newtheorem{example}[theorem]{Example}

\numberwithin{equation}{section} \catcode`@=11

\begin{document}
\maketitle

\begin{abstract}
Let $L_t:=\Delta_t+Z_t$ for a $C^{1,1}$-vector field $Z$ on a  differential manifold $M$ with possible boundary $\partial M$,
where $\Delta_t$ is the Laplacian induced by a time dependent metric $g_t$ differentiable in $t\in [0,T_c)$.
We first introduce the damp gradient operator, defined on the path space  with reference measure $\mathbb{P}$, the law of the
(reflecting) diffusion process generated by $L_t$ on the  base manifold; then establish
the integration by parts formula for underlying directional derivatives and prove the log-Sobolev inequality for the associated Dirichlet form, which is further applied to the free path spaces; and finally, establish
 numbers of transportation-cost inequalities associated to the uniform distance, which are equivalent to the curvature lower bound  and the convexity of the  boundary.

\end{abstract}



\def\dint{\displaystyle\int}
\def\ct{\cite}
\def\lb{\label}
\def\ex{Example}
\def\vd{\mathrm{d}}
\def\dis{\displaystyle}
\def\fin{\hfill$\square$}
\def\thm{theorem}
\def\bthm{\begin{theorem}}
\def\ethm{\end{theorem}}
\def\blem{\begin{lemma}}
\def\elem{\end{lemma}}
\def\brem{\begin{remark}}
\def\erem{\end{remark}}
\def\bexm{\begin{example}}
\def\eexm{\end{example}}
\def\bcor{\bg{corollary}}
\def\ecor{\end{corollary}}
\def\r{\right}
\def\l{\left}
\def\var{\text {\rm Var}}
\def\lmd{\lambda}
\def\alp{\alpha}
\def\gm{\gamma}
\def\Gm{\Gamma}
\def\e{\operatorname{e}}
\def\gap{\text{\rm gap}}
\def\dsum{\displaystyle\sum}
\def\dsup{\displaystyle\sup}
\def\dlim{\displaystyle\lim}
\def\dlimsup{\displaystyle\limsup}
\def\dmax{\displaystyle\max}
\def\dmin{\displaystyle\min}
\def\dinf{\displaystyle\inf}
\def\be{\begin{equation}}
\def\de{\end{equation}}
\def\dint{\displaystyle\int}
\def\dfrac{\displaystyle\frac}
\def\zm{\noindent{\bf  Proof.\ }}
\def\endzm{\quad $\Box$}
\def\mO{\mathcal{O}}
\def\mW{\mathcal{W}}
\def\mL{\mathcal{L}}
\def\LC{\mathcal{L}{\rm Cut}}
\def\proclaim#1{\bigskip\noindent{\bf #1}\bgroup\it\  }
\def\endproclaim{\egroup\par\bigskip}
\baselineskip 18pt

\section{Introduction}
Let $M$ be a $d$-dimensional differential manifold with possible boundary $\partial M$ equipped with a complete
Riemannian metric $(g_t)_{t\in [0,T_c)}$, which is $C^1$ in $t$. For simplicity, we take the notations: for $X, Y\in TM$,
\begin{align*}
  &{\rm Ric}_t^Z(X,Y):={\rm Ric}_t(X,Y)-\l<\nabla^t_XZ_t, Y\r>_t,\ \
   \mathcal{G}_t(X,X):=\partial_tg_t(X,X)\\
    &\mathcal{R}_t^{Z}(X,Y):={\rm Ric}_t^Z(X,Y)-\frac{1}{2}\mathcal{G}_t(X,Y),
\end{align*}
 where ${\rm Ric}_t$ is the Ricci curvature tensor with respect to $g_t$, $(Z_t)_{t\in [0,T_c)}$ is a $C^1$-family of vector fields, and $\l<\cdot,\cdot\r>_t:=g_t(\cdot,\cdot)$. Define the second fundamental form of the boundary by
 $$\mathbb{I}_{t}(X,Y)=-\l<\nabla^t_XN_t,Y\r>_t, \ \ X,Y\in T\partial M,$$
 where  $N_t$ is the
inward unit normal vector field of the boundary associated with $g_t$; $T\partial M$ is the tangent space of $\partial M$.
Consider the elliptic operator $L_t:=\Delta_t+Z_t$.
Let $X_t$ be the reflecting inhomogeneous diffusion process generated by $L_t$ (called reflecting $L_t$-diffusion process). Assume that $X_t$ is non-explosive before $T_c$.
 In this case, for any $0\leq S<T<T_c$, the distribution $\Pi^{S,T}$ of $X_{[S,T]}:=\{X_t: t\in [S,T]\}$
 is a probability measure on the path space
 $$W^{S,T}:=C([S,T]; M),$$
 when $S=0$, we write $W^T:=W^{0,T}$ and $\Pi^T:=\Pi^{0,T}$ for simplicity.
  For each point $x\in M$, let $X_{[0,T]}^x=\{X_t^x: 0\leq t\leq T< T_c\}$ and $W_x^T=\{\gamma\in
W^T:\gamma_0=x\}$.
 It has been well-known that there is a strong connection between
the behavior of the distribution of the $M$-valued Brownian motion
associated with metric $g$ and the geometry of their underlying
space.
 This paper is devoted to further the study of this relation  over a inhomogeneous manifold. Note that although our discussions base on the manifold with boundary, the results are also new for the manifold without boundary.

When  the metric is independent of $t$, the theory about stochastic analysis on the path space
over a complete Riemannian manifold  has been well
developed since Driver \cite{Dr} proved the quasi-invariance theorem for
the Brownian motion. Then,
integration by parts formula for the associated gradient operator,
induced by the quasi-invariant flow,  was established, which leads to
the study of the functional inequalities with respect to the corresponding
Dirichlet form (see e.g. \cite{ Hsu97, FW}).  For  the case with boundary, see \cite{Hsu02, W11b} for the corresponding results on manifold with
boundary; see \cite{Bismut,CM, FM} for an intertwining formula for the
differential of It\^{o} development map.

A probabilistic approach to these problem was initiated by  Abandon et al, who  constructed  $g_t$-Brownian motion on
time-inhomogeneous space  in \cite{ACT}.
 Recently,  Chen \cite{Ch} gives the interwining formula
and log-Sobolev inequality for usual Dirichlet form (without damp) on the path space over time-inhomogeneous manifold.
   The  main purpose of
this paper is to prove the integration by parts formula and further
establish the log-Sobolev inequality w.r.t. the associated Dirichlet
form, defined by the damped gradient operator. Note that from technical point of view, our method relies on \cite{Hsu02,W11b}.

Our second purpose is to discuss the Talagrand type
transportation-cost inequalities on the path space with respect to the
uniform distance  on  time-inhomogeneous space. In 1996, Talagrand \cite{T96} found that the $L^2$-Wasserstein distance to
the standard Gaussian measure can be dominated by the square root of
twice relative entropy on  $M=\mathbb{R}^{d}$ with constant metric. This inequality  has been extended to
distributions on finite- and infinite-dimensional spaces. In
particular, this inequality was established on the path space of
diffusion processes with respect to several different distances
(i.e. cost functions). See  \cite{FU} for the details on the Wiener
space with the Cameron-Martin distance; see \cite{W02,DGW} on the path space
of diffusions with the $L^2$-distance; see \cite{W04}  on the Riemannian path
space with intrinsic distance induced by the Malliavin gradient
operator, and \cite{FWW,WZ,W09} on the path space of diffusions with the
uniform distance.

The rest parts of the paper are organized as follows. In the following two
sections, we  construct the Hsu's multiplicative functional
and then define the corresponding damped gradient operator, which
satisfies an integration by parts formula induced by intrinsic
quasi-invariant flows. In Section 4,  the log-Sobolev inequality for the associated
Dirichlet form is established and extends to the free path space. In Section 5, some transportation-cost inequalities are presented to be equivalent to the curvature condition and the convexity of
the boundary, and parts of these are extended to non-convex case in finial section.

\section{(Reflecting) $L_t$-diffusion process  and multiplicative functional}
Let $\mathcal{F}(M)$  be the frame bundle over $M$ and
$\mathcal{O}_{t}(M)$ be
the orthonormal frame bundle over $M$ with respect to $g_t$.  Let $\mathbf{p}:
\mathcal{F}(M)\rightarrow M$ be the projection
from $\mathcal{F}(M)$ onto $M$.
Let $\{e_{\alpha}\}_{\alpha=1}^{d}$ be the canonical orthonormal basis of
$\mathbb{R}^d$. For any $u\in \mathcal{F}(M)$,
let $H^{t}_{i}(u)$ be the $\nabla^{t} $  horizontal lift
of $ue_i$ and $\{V_{\alpha, \beta}(u)\}_{\alpha,\beta=1}^d$ be the canonical basis
of vertical  fields over $\mathcal{F}(M)$, defined by $V_{\alpha,\beta}(u)=Tl_u(\exp(E_{\alpha,\beta}))$,
where $E_{\alpha,\beta}$ is the canonical basis of $\mathcal{M}_d(\mathbb{R})$, the $d\times d$ matric space over $\mathbb{R}$,
 and $l_u:Gl_d(\mathbb{R})\rightarrow \mathcal{F}(M)$ is the left multiplication from the general linear group to $\mathcal{F}(M)$, i.e.
   $l_u\exp(E_{\alpha,\beta})=u\exp(E_{\alpha,\beta})$.

 Let $B_t:=(B_t^1,B_t^2,\cdots,B_t^d)$ be a $\mathbb{R}^d$-valued Brownian motion on  a complete
filtered probability space $(\Omega,\{\mathscr{F}_t\}_{t\geq 0}, \mathbb{P})$  with the natural filtration
$\{\mathscr{F}_t\}_{t\geq 0}$. Assume the
elliptic generator $L_t$ is a $C^{1}$ functional of time with
associated metric $g_t$:
$$L_t=\Delta_t +Z_t $$
where  $Z_t $ is a $C^{1,1}$ vector field on $M$. As in the time-homogeneous case, to construct the $L_t$-diffusion process,
we first construct the corresponding horizontal diffusion process
generated by $L_{\mathcal{O}_t(M)}:=\Delta _{\mathcal {O}_t(M)}+H_{Z_t }^t$
  by solving the Stratonovich stochastic
diffusion equation (SDE):
$$\begin{cases}
 \vd u_t=\sqrt{2}\dsum_{i=1}^{d}H_{i}^t(u_t)\circ \vd
B_t^{i}+H_{Z_t}^t(u_t)\vd t-\frac{1}{2}\dsum_{i,j}\partial_tg_t(u_te_i,u_te_j)V_{i, j}(u_t)\vd t+H^t_{N_t}(u_t)\vd l_t,\medskip\\
u_0\in \mathcal{O}_0(M),
\end{cases}$$
where  $\Delta_{\mathcal{O}_t(M)}$ is the horizontal Laplace operator on $\mathcal{O}_t(M)$;
$H_{Z_t}^t$ and $H^t_{N_t}$ are the $\nabla^t$
horizontal lift of  $Z_t$ and $N_t$ respectively; $l_t$ is an increasing
process supported on $\{t\geq 0:X_t:={\bf p}u_t\in \partial M\}$.
By a similar discussion as in \cite[Proposition 1.2]{ACT}, we see that
the last term promises $u_t\in \mathcal{O}_t(M)$.
Since $(H_{Z_t}^t)_{t\in [0,T_c)}$ is a $C^{1,1}$-family vector field, it is well-known that (see
e.g. \cite{Hsu}) the equation has a unique solution up to the life time
 $\zeta:=\dlim_{n\rightarrow
\infty}\zeta_n$, and
$$\zeta_n:=\inf\{t\in [0,T_c):\rho_t({\bf p}u_0, {\bf p}u_t)\geq n\}, \ n\geq 1,\ \ \inf\varnothing=T_c,$$
where $\rho_t(x,y)$ is the distance between $x$ and $y$ associated with $g_t$. Let $X_t={\bf p}u_t$. It is easy to see that $X_t$
solves the equation $$\vd X_t=\sqrt{2}u_t\circ \vd B_t+Z_t(X_t)\vd
t+N_t(X_t)\vd l_t,\ \ X_0=x:={\bf p}u$$ up to the life time $\zeta$.
By the It\^{o} formula, for any $f\in C_0^{1,2}([0,T_c)\times M)$ with
$N_tf_t:=N_tf_t|_{\partial M}=0$,
$$f(t,X_t)-f(0,x)-\int_0^t\l({\partial_ s}+L_s\r)f(s,X_s)\vd
s=\sqrt{2}\int_0^t\l<u_s^{-1}\nabla^s f(s,\cdot)(X_s),\vd B_s\r>$$
is a
martingale up to the life time $\zeta$. Here and what follows, we denote the inner product on $\mathbb{R}^d$ by $\l<\cdot,\cdot\r>$, and  write $f(t,\cdot)=f_t$ for simplicity.  So, we call $X_t$  the
reflecting diffusion process generated by $L_t$. When $Z_t\equiv
0$, then $\tilde{X}_t:=X_{t/2}$ is generated by
$\frac{1}{2}\Delta_t$ and is called the reflecting $g_t$-Brownian
motion on $M$. In what follows, we assume the process is non-explosive.

\subsection{Multiplicative functional}
To construct the desired continuous multiplicative functional, we
need the following assumption.
\begin{enumerate}
  \item [{\bf(A)}] There exist two constants $K, \sigma \in
C([0,T_c))$ such that $$\mathcal{R}^Z_t\geq K(t),\ \mathbb{I}_t\geq \sigma(t)\ \mbox{and}\ \
\mathbb{E}e^{\lambda\int_0^t\sigma^-(s)\vd l_s^x}<\infty$$ holds for
$\lambda>0$, $t\in [0,T_c)$, and $x\in M$, where $\sigma ^-=0\vee(-\sigma)$.
\end{enumerate}

 To introduce Hsu's discontinuous multiplicative functional, we need the
lift operators of $\mathcal{R}^{Z}_t$, ${\mathcal{G}}_t$,
${\mathbb{I}}_t$. For any $u\in \mathcal{O}_t(M)$, let
$\mathcal{R}_u^Z(t)(a,b)=\mathcal{R}^{Z}_t(ua,ub)$ and
$\mathcal{G}_u(t)(a,b)=\mathcal{G}_t(ua,ub),\  a,b \in \mathbb{R}^d$. Let ${\bf p}_{\partial}: TM\rightarrow T\partial M$ be the orthogonal
projection at points on $(\partial M,g_t)$.  For
any $u \in \mathcal{O}_t(M)$ with ${\bf p}u\in \partial
M$, let
$$\mathbb{I}_u(t)(a, b)=\mathbb{I}_t({\bf p}_{\partial}ua, {\bf p}_{\partial}ub), a, b\in
\mathbb{R}^d.$$
 For $u\in \partial \mathcal{O}_t(M)$, the boundary of $\mathcal{O}_t(M)$,
let $${P}_u(t)(a, b)=\l<ua,N_t\r>_t\l<ub, N_t\r>_t, a,b\in
\mathbb{R}^d.$$

 For any $\varepsilon>0$ and $r\geq 0$, let $Q^{x,\varepsilon}_{r,t}$ solve the following SDE on $\mathbb{R}^d\otimes\mathbb{R}^d$:
\begin{align}\label{MF}\vd
Q^{x,\varepsilon}_{r,t}=-Q^{x,\varepsilon}_{r,t}\l\{\mathcal{R}_{u^x_t}^Z(t)\vd
t+(\varepsilon^{-1}P_{u_t^x}(t)+\mathbb{I}_{u_t^x}(t))\vd
l^x_t\r\}.\end{align}
When the metric is independent of $t$ and
 $M$ is compact, letting
$\varepsilon\downarrow 0$, the process $Q^{x,\varepsilon}_{r,t}
$ converges in $L^2$ to an adapted right-continuous process
$Q^x_{r,t}$ with left limit, such that $Q^x_{r,t}P_{u^x_t}(t)=0$ if
$X^x_t\in \partial M$ (see  \cite[Theorem 3.4]{Hsu02}). Here, we follow Wang \cite[Theorem 4.1.1]{Wbook2}, and introduce a slightly different but
simpler construction of the multiplicative functional by solving a
random integral equation on $\mathbb{R}^d\otimes \mathbb{R}^d$.
\begin{theorem}\label{s1-1}
Assume ${\bf(A)}$, then
\begin{enumerate}
  \item [$(1)$] for any $x \in M$, $0\leq r \leq t< T_c$ and $u_0^x\in
  \mathcal{O}_0(M)$, the equation
$$Q^x_{r,t}=\l(I-\int_r^tQ_{r,s}^x\mathcal{R}^Z_{u_s^x}(s)\vd s-\int_s^t Q_{r,s}^x\mathbb{I}_{u_s^x}(s)\vd l^x_s\r)(I-{ 1}_{\{X_t^x\in \partial M\}}P_{u_t^x}(t))$$
has a unique solution;
  \item [$(2)$] for any $0\leq r\leq t< T_c$, $\|Q_{r,t}^x\|\leq e^{-\int_r^t K(s)\vd s-\int_r^t\sigma(s)\vd l^x_s}$ a.s., where $\|\cdot\|$ is the operator norm
for $d\times d$-matrices;
  \item [$(3)$] for any $0 \leq r \leq s \leq t< T_c$, $Q^x_{
r,t}= Q^x_{r,s}Q^x_{s,t} $ a.s..
\end{enumerate}
\end{theorem}
\begin{proof}
 We only consider the existence of the solution up to a arbitrarily
given time $T\in (r,T_c)$, since the uniqueness is obvious. In
the following, we drop the superscript $x$ for simplicity. By the assumption
$\mathbf{(A)}$ and (\ref{MF}), we have
\begin{align*}
\|Q_{r,t}^{\varepsilon}\|^2\leq
1-2\int_r^t\|Q_{r,s}^{\varepsilon}\|^2K(s)\vd
s-2\int_r^t\|Q_{r,s}^{\varepsilon}\|^2\sigma(s)\vd
l_s-\frac{2}{\varepsilon}\int_s^t\|Q_{r,s}^{\varepsilon}P_{u_s}(s)\|^2\vd
l_s, \ \ t>r.
\end{align*}
Therefore, we obtain $\|Q_{r,t}^{\varepsilon}\|^2\leq e^{-2\int_r^tK(s)\vd s-2\int_r^t\sigma(s)\vd l_s},\ \ t\geq r$ and
\begin{align}\label{s1-2}
\int_r^T\|Q^{\varepsilon}_{r,s}P_{u_s}(s)\|^2\vd l_s\leq
\frac{\varepsilon}{2}\l[1+2\l(\int_r^TK^-(s)\vd
s+\int_r^T\sigma^-(s)\vd l_s\r)e^{2\int_r^TK^-(s)\vd
s+\int_r^T\sigma^-(s)\vd l_s}\r].
\end{align}
Combining this with ${\bf(A)}$, we obtain
\begin{align}\label{s1-3}
\lim_{\varepsilon \rightarrow
0}\mathbb{E}\int_r^T\|Q_{r,s}^{\varepsilon}P_{u_s}(s)\|^2\vd l_s=0
\end{align}
and
$$\sup_{\varepsilon\in(0,1)}\mathbb{E}\int_r^T\|Q_{r,t}^{\varepsilon}\|^2(\vd t+ \vd l_t)<\infty.$$
Because of the latter, we may select a sequence
$\varepsilon_n\downarrow 0$ and an adapted process
$\overline{Q}_{r,\cdot}\in L^2(\Omega \times [r,T])\rightarrow
\mathbb{R}^d\otimes \mathbb{R}^d; \mathbb{P}\times (\vd t+\vd l_t)$,
such that for any $g\in L^2(\Omega\times [r,T])\rightarrow
\mathbb{R}^d; \mathbb{P}\times (\vd t+ \vd l_t)$,
$$\lim_{n\rightarrow \infty}\mathbb{E}\int_r^TQ^{\varepsilon_n}_{r,t}g_t(\vd t+\vd l_t)=\mathbb{E}\int_r^T\overline{Q}_{r,t}g_t(\vd t+\vd l_t).$$
The following discussion is almost the same as the
case with constant metric, see \cite{W11b} for details.
\end{proof}

 The following result is a direct conclusion of Theorem \ref{s1-1}.
\begin{proposition}\label{s1-p1}
 Assume {\bf(A)}. For any $\mathbb{R}^d$-valued continuous
semi-martingale $h_t$ with $$1_{\{X^x_t\in \partial
M\}}P_{u_t}(t)h_t=0,$$
$$\vd Q^x_{r,t}h_t=\overline{Q}^x_{r,t}\vd h_t-Q^x_{r,t}\mathcal{R}_{u_t^x}^Z(t)h_t\vd t-Q^x_{r,t}\mathbb{I}_{u_t^x}(t) h_t\vd l^x_t,\ \  t\geq r,$$
where
$$\overline{Q}^x_{r,t}=\l(I-\int_r^tQ^x_{r,s}\mathcal{R}^Z_{u_s^x}(s)\vd s-\int_r^tQ^x_{r,s}\mathbb{I}_{u_s^x}(s)\vd l_s^x\r).$$
\end{proposition}
\begin{proof}
As $1_{\{X^x_t\in \partial M\}}P_{u_t}(t)h_t=0,$ and by Theorem
\ref{s1-1}, we have
$$Q^x_{r,t}h_t=\l(I-\int_r^tQ_{r,s}^x\mathcal{R}_{u_s^x}^Z(s)\vd s+\frac{1}{2}\int_r^tQ_{r,s}^x\mathcal{G}_{u_s^x}(s)\vd s-\int_r^tQ_{r,s}\mathbb{I}_{u_s^x}(s)\vd l_s^x\r)h_t=\overline{Q}_{r,t}^xh_t.$$
Then the proof is completed by using It\^{o} formula.
\end{proof}

Recall that $\{P_{r,t}\}_{0\leq r\leq t< T_c}$ is the Neumann  semigroup
generated by $L_t$.  The following is a consequence of Proposition
\ref{s1-p1}, which is the derivative  formula of the diffusion
semigroup, known as Bismut-Elworthy Li formula (see e.g. \cite{Bismut,EL}).
\begin{corollary}\label{s1-c1}
Assume ${\bf(A)}$. Let $f\in C^{\infty}_b(M)$. Then for any $0\leq r<t<T_c$,
$$ [r, t]\ni s\rightarrow Q^x_{r,s}(u^x_s)^{-1}\nabla ^sP_{s,t}f(X^x_s)$$ is a martingale.
Consequently,
\begin{align}\label{Bismut}(u_r^x)^{-1}\nabla^rP_{r,t}f(X_r^x)=\mathbb{E}(Q^x_{r,t}(u_t^x)^{-1}\nabla^tf(X_t^x)|\mathscr{F}_r),\end{align}
and for any adapted $\mathbb{R}_+$-valued precess $\xi$ satisfying  $\xi(r)=0, \xi(t)=1$, and $\mathbb{E}(\int_r^t\xi'(s)^2\vd s)^{\alpha}<\infty$ for $\alpha>1/2$, there holds
$$(u_r^x)^{-1}\nabla^rP_{r,t}f(X_r^x)=\frac{1}{\sqrt{2}}\mathbb{E}\l(f(X_t^x)\int_r^t\xi'(s)(Q_{r,s}^x)^*\vd B_s\big|\mathscr{F}_r\r).$$
\end{corollary}
\begin{proof}
The proof is essentially due to \cite[Corollary 3.4]{W11b}. Without losing generality, we assume $r=0$ and drop the
superscript $x$ for simplicity.

We first prove that $Q_sh_s$ is a martingale. Let
$h_s=(u_s)^{-1}\nabla^{s}f(X_s)$. Since $\nabla^{s}P_{s,t}f$ is
vertical to $N_s$ on $\partial M$, $1_{\{X_s\in \partial
M\}}P_{u}(s)h_s=0$. Then  we have
$$\vd Q_sh_s=\overline{Q}_s\vd h_s-Q_s\mathcal{R}_u^Z(s)h_s\vd s-Q_s\mathbb{I}_u(s)h_s\vd l_s.$$
Let
$F(u,s):=u^{-1}\nabla^{s}P_{s,t}f(\mathbf{p}u),\ u\in \mathcal{O}_s(M)$, then
$$\frac{\vd}{\vd s}F(u,s)=-u^{-1}\nabla^{s}P_{s,t}f({\bf p}u)=-L_{\mathcal{O}_s(M)}F(\cdot,s)(u)+(\mathcal{R}_u^Z(s)+\frac{1}{2}\mathcal{G}_u(s))F(u,s), s\in [0,t].$$
On the other hand, noting that
$$\vd u_t=\sqrt{2}\sum_{i=1}^{d}H_i^t\circ \vd B_t^{i}+H^t_{Z_t}(u_t)\vd t-\frac{1}{2}\sum_{\alpha,\beta=1}^{d}\mathcal{G}_{\alpha,\beta}(t,u_t)V_{\alpha,\beta}(u_t)\vd t+H^t_{N_t}(u_t)\vd l_t.$$
By the It\^{o} formula, we have
\begin{align}\label{Ieq}
\vd F(u_s, t_0)=&\vd M_s+L_{\mathcal{O}_s{M}}F(\cdot,
t_0)(u_s)\vd s+H_{N_s}^sF(\cdot, t_0)(u_s)\vd l_s\nonumber\\
&
-\frac{1}{2}\sum_{\alpha,\beta}\mathcal{G}_{\alpha,\beta}(s,u_s)V_{\alpha,\beta}(u_s)F(\cdot,
t_0)(u_s)\vd s.
\end{align}
where $$\vd M_s:=\sqrt{2}\sum_{i=1}^{d}H^s_{i}F(\cdot,
t_0)(u_s)\vd B_s^i.$$ Therefore,
$$\vd h_s=\vd M_s+\mathcal{R}_u^Z(s)h_s\vd s +H_{N_s}^sF(\cdot, s)(u_s)\vd l_s.$$
Since $1_{\{X_s\in \partial M\}}Q_sP_{u_s}(s)=0$, combining this with (\ref{Ieq}), we
obtain
$$\vd Q_sh_s=Q_s\vd M_s + Q_s(I-P_{u_s}(s))\l\{H_{N_s}^sF(\cdot, s)(u_s)-\mathbb{I}_{u}{(s)}F(u_s, s)\r\}\vd l_s.$$
Noting that for any $e\in \mathbb{R}^d$, it follows from  that when
$X_s\in \partial M$,
\begin{align*}
\l<(I-P_{u_s}(s))H^s_{N_s}F(\cdot, s)(u_s),e\r>&={\rm
Hess}^s_{P_{s,t}f}(N_s, \mathbf{p}_{\partial}u_se)
=\mathbb{I}_s(\nabla^{s}P_{s,t}f(X_s),
\mathbf{p}_{\partial}u_se)\\
&=\mathbb{I}_{u}(s)(F(u,s),e)=\l<\mathbb{I}_{u}(s)F(u,s),e\r>,
\end{align*}
we conclude that
$$(1-P_u(s))\{H^s_{N_s}F(\cdot, s)(u_s)-\mathbb{I}_{u}(s)F(u_s,s)\}\vd l_s=0.$$
Therefore, $Q_sh_s$ is a local martingale. By ({\bf A}), it is then a martingale
according to \cite[Theorem 4.9]{Cheng}.

The following step is similar as shown in step (b) in the proof of
\cite[Lemma 3.3]{W11b}. We skip it here.

\end{proof}
\section{Damped Gradient, quasi-invariant flows and integration
by Parts}
 When the metric is independent of $t$, the Malliavin derivative can be realized by
quasi-invariant flows  for diffusion on manifolds (see e.g. \cite{Hsu, W11b}). In this section,
by using the multiplicative functional constructed in \S 2.1, we first introduce the damped gradient operator as in \cite{FM}, then introduce quasi-invariant flows
induced by SDEs with refection, and finally link them by
establishing an integration by parts formula.
\subsection{Damped Gradient operator and quasi-invariant flows}
We shall use multiplicative functionals $\{Q^x_{r,t}: 0\leq r\leq t
< T_c\}$ to define the damped gradient operator for functionals of
$X^x$.

Let
$$\mathscr{F}C_0^{\infty}=\{W^T\ni\gamma \rightarrow f(\gamma_{t_1},\gamma_{t_2},\cdots, \gamma_{t_n}): n\geq 1, 0<t_1<t_2<\cdots < t_n \leq T, f \in C_0^{\infty}(M^n)\}$$
be the class of smooth cylindrical functions on $W^T$.  Let
$${\bf H}_0:=\l\{h\in C([0,T];\mathbb{R}^d): h(0)=0, \|h\|_{{\bf H}_0}:=\int_0^T|\dot{h}(s)|^2\vd s<\infty\r\}$$
be the Cameron-Martin space on the flat path space.
For any $F\in \mathscr{F}C^{\infty}_0$ with
$F(\gamma)=f(\gamma_{t_1},\gamma_{t_2},\cdots, \gamma_{t_n})$,
define the damped gradient $D^0F(X_{[0,T]}^x)$ as an ${\bf
H}_0$-valued random variable by setting $(D^0F(X^x_{[0,T]}))(0)=0$
and
$$\frac{\vd }{\vd t}(D^0F(X^x_{[0,T]}))(t)=\sum_{i=1}^{n}1_{\{t<t_i\}}Q^x_{t,t_i}(u^x_{t_i})^{-1}\nabla^{t_i}_if(X_{t_1}^x,X_{t_2}^x,\cdots,X_{t_n}^x), \ t\in [0,T].$$
where $\nabla^{t_i}_{i}$ denotes the gradient operator w.r.t. the
$i$-th component associated with $g_{t_i}$. Then, for any ${\bf
H}_0$-valued random variable $h$, let
$$D^0_hF(X_{[0,T]}^x)=\l<D^0 F(X_{[0,T]}), h\r>_{{\bf H_0}}=\sum_{i=1}^{n}\int_0^{t_i}\l<(u_{t_i}^x)^{-1}\nabla^{t_i}_if(X_{t_1}^x, X_{t_2}^x,\cdots, X_{t_n}^x),(Q_{t,t_i}^x)^*h'(t)\r>\vd t,$$
We would like to indicate that the formulation of $D^0_hF$ is consistent with \cite{FM} for the case with constant metric.
Note that compared with usual gradient operator, it contains
 the multiplicative functional, which  affects the log-Sobolev constant.
This operator links $D^0_hF$ to the directional derivative
induced by a quasi-invariant flow. We now turn to investigating this relation.  The main idea essentially due to \cite{Hsu02},  where
quasi-invariant flows are constructed for constant manifold $M$ with boundary
$\partial M$. Let ${\bf \tilde{H}_0}$ be the set of all adapted
elements in $L^2(\Omega\rightarrow \mathbf{H}_0;
\mathbb{P})$; i.e.
$${\bf \tilde{H}_0}=\{h\in L^2(\Omega \rightarrow {\bf {H}_0}; \mathbb{P}): h(t)\  \mbox{is\  a}\  \mathscr{F}_t\mbox{-measurable}, t\in [0,T]\}.$$
Then, ${\bf \tilde{H}_0}$ is a Hilbert space with inner product
 $$\big<h,\tilde{h}\big>_{{\bf \tilde{H}_0}}:=\mathbb{E}\big<h,\tilde{h}\big>_{{\bf H_0}}=\mathbb{E}\int_0^T\big<h'(t),\tilde{h}'(t)\big>\vd t, h,\  \tilde{h}\in {{\bf \tilde{H}_0}}.$$
 For $h\in
\mathbf{\tilde{H}_0}$ and $\varepsilon >0$, let $X^{\varepsilon,
h}_t$ solve the SDE
\begin{align}\label{2s-e1}
\vd X_t^{\varepsilon, h}=\sqrt{2} u_t^{\varepsilon, h}\circ \vd
B_t+Z_t( X_t^{\varepsilon, h})\vd t++\varepsilon\sqrt{2}u_t^{\varepsilon, h}
h'(t)\vd t+N_t( X_t^{\varepsilon, h})\vd
l_t^{\varepsilon, h},
\end{align}
where $l_t^{\varepsilon, h}$  and $u_t^{\varepsilon, h}$ are,
respectively, the local time on $\partial M$ and the horizontal lift
on $\mathcal{O}_t(M)$ for $X_t^{\varepsilon, h}$. The detailed construction of $X^{\varepsilon,
h}_t$ is similar as in Section 2.  To see that
$\{X_{[0,T]}^{\varepsilon, h}\}_{\varepsilon\geq 0}$ has the flow
property also in our setting, let
$$\Theta: W_0:=\{\omega \in C([0,T]; \mathbb{R}^d): \omega _0=0\}\rightarrow W^T$$
be measurable such that $X=\Theta (B)$, $B\in W_0$. For any $\varepsilon >0$ and
a function $\Phi:W_0\rightarrow W^T$, let
$(\theta^h_{\varepsilon}\Phi)(\omega)=\Phi(\omega +\varepsilon h)$.
Then $X_{[0,T]}^{\varepsilon, h}=(\theta ^h_{\varepsilon}\Theta)(B),\
\varepsilon>0$. Hence,
$$X^{\varepsilon_1+\varepsilon_2, h}_{[0,T]}=\theta^h_{\varepsilon_1}X^{\varepsilon_2, h}_{[0,T]},\  \varepsilon_1, \varepsilon_2\geq 0.$$
Moreover, let us explain that the flow is quasi-invariant, i.e. for
each $\varepsilon\geq 0$, the distribution of $X^{\varepsilon,
h}_{[0,T]}$ is absolutely continuous w.r.t. that of $X^x_{[0,T]}$.
Let
$$R^{\varepsilon, h}=\exp{\l[\varepsilon \int_0^T\l<h'(t), \vd B_t\r>-\frac{\varepsilon^2}{2}\int_0^T|h'(t)|^2\vd t\r]}.$$
By the Girsanov theorem
 $$B_t^{\varepsilon, h}:=B_t-\varepsilon h(t)$$
 is the $d$-dimensional Brownian motion under the probability $R^{\varepsilon, h}\mathbb{P}$. Thus, the distribution
of $X^x_{[0,T]}$ under $R^{\varepsilon, h}$ coincides with that of
$X^{\varepsilon,h}_{[0,T]}$ under $\mathbb{P}$. Therefore, the map
$X^x_{[0,T]}\rightarrow X^{\varepsilon,h}_{[0,T]}$ is
quasi-invariant. The quasi-invarient property leads us to prove the
following property.
\begin{proposition}\label{s2-l1}
Let $x\in M$ and $F\in \mathscr{F}C^{\infty}$. Then
\begin{align*}
\lim_{\varepsilon \downarrow 0}\mathbb{E}\frac{F(X^{\varepsilon,
h}_{[0,T]})-F(X_{[0,T]}^x)}{\varepsilon}=\mathbb{E}\l\{F(X_{[0,T]}^x)\int_0^T\l<h'(t),\vd
B_t\r>\r\}
\end{align*}
holds for $h\in {\bf\tilde{H}}_{0,b}$, the set of all elements in ${\bf
\tilde{H}}_{0}$ with bounded $\|h\|_{\tilde{\mathbf{H}}_0}$.
\end{proposition}
\begin{proof}
As we have explained that $B_t^{\varepsilon, h}=B_t-\varepsilon
h(t)$ is a $d$-dimensional Brownian motion under
$R^{\varepsilon,h}\mathbb{P}$. By the weak uniqueness of
(\ref{2s-e1}), we conclude that the distribution of $X^x$ under
$R^{\varepsilon, h}\mathbb{P}$ coincide with that of
$X^{\varepsilon, h}$ under $\mathbb{P}$. In particular,
$\mathbb{E}F(X_{[0,T]}^{\varepsilon, h})=\mathbb{E}[R^{\varepsilon,
h}F(X_{[0,T]}^x)]$. Thus, the assertion follows from $\frac{\vd
R^{\varepsilon,h}}{\vd \varepsilon}|_{\varepsilon=0}$ and the
dominated convergence theorem since $\{R^{\varepsilon,
h}\}_{\varepsilon \in [0,1]}$ is uniformly integrable for  $h\in
\tilde{{\bf H}}_{0,b}$.
\end{proof}
\subsection{ Integration by parts formula}
 In this section, an integration by parts formula for $D^0_hF$ is established and  applied to clarifying the link between this formula and
the derivative induced by the flow $\{X^{\varepsilon,h}_{[0,T]}
\}_{\varepsilon \geq 0}.$ The main result of this subsection is presented as follows.
\begin{theorem}\label{s2-t1}
Assume ${\bf(A)}$. For any $x\in M$ and $F\in
\mathscr{F}C_0^{\infty}$,\begin{align}\label{In} \lim_{\varepsilon \downarrow
0}\mathbb{E}\frac{F(X^{\varepsilon,h}_{[0,T]})-F(X^x_{[0,T]})}{\varepsilon}=\mathbb{E}\{D^0_hF\}(X^x_{[0,T]})=\mathbb{E}\l\{F(X^{x}_{[0,T]})\int_0^T\l<h'(t),
\vd B_t\r>\r\} \end{align} holds for all $h\in {\bf
\tilde{H}}_{0,b}$.
\end{theorem}
\begin{proof} By the Proposition \ref{s2-l1}, it is sufficient for us to prove the second equality. The proof is similar to the constant metric case (see \cite[Theorem 2.1]{W11b}) due to the following Lemmas \ref{s2-l2} and \ref{s2-l3} and the Markov property, . Note that the second equality holds for all $h\in \tilde{{\bf H}}_0$.
\end{proof}
The following lemma gives the gradient formula
for special cylinder functions.
\begin{lemma}\label{s2-l2}
For any $n\geq 1$, $0<t_1<\cdots <t_n\leq T$, and $f\in
C^{\infty}(M^n)$,
\begin{align*}
(u_{t_1}^x)^{-1}\nabla^{t_1}_{1}\mathbb{E}\{f(X_{t_1}^x,X_{t_2}^x,\cdots,
X_{t_n}^x)|\mathscr{F}_{t_1}\}=\sum_{i=1}^{n}\mathbb{E}\{Q_{t_1,t_i}^x(u_{t_i}^x)^{-1}\nabla^{t_i}_{i}f(X_{t_1}^x,X_{t_2}^x,\cdots,
X_{t_n}^x)| \mathscr{F}_{t_1}\}
\end{align*}
holds for all $x\in M$ and $u_0^x\in \mathcal{O}_0(M)$, where
$\nabla^{t_i}_{i}$ denotes the $g_{t_i}$-gradient w.r.t. the $i$-th
component.
\end{lemma}
\begin{proof}
It is obvious that the assertion is true for $n=1$. By  (\ref{Bismut}), we have
$$(u_{t_1}^x)^{-1}\nabla^{t_1}\mathbb{E}(f(X_{t_2})|\mathscr{F}_{t_1})=\mathbb{E}(Q_{t_1,t_2}^x(u_{t_2}^x)^{-1}\nabla_2^{t_2}f(X_{t_2})|\mathscr{F}_{t_1}).$$
which plus the case of $n=1$, we prove the result for $n=2$. Assume
that it holds for $n=k$, $k\geq 2$. It remains to prove the case for
$n=k+1$. To this end, by Markov property, set
$$g(X_{t_1}^x, X_{t_2}^x)=\mathbb{E}f(X_{t_1}^x, X_{t_2}^x, \cdots, X_{t_{k+1}}^x|\mathscr{F}_{t_{2}})$$
By the assumption for $n=k$, we have
\begin{align}\label{s2-e2}
&(u_{t_1}^x)^{-1}\nabla^{t_1}_1\mathbb{E}\l\{g(X_{t_1},
X_{t_2})|\mathscr{F}_{t_1}\r\}\nonumber\\
=&\mathbb{E}\l\{(u_{t_1}^x)^{-1}\nabla_1^{t_1}g(X_{t_1}^{x},X_{t_2}^x)|\mathscr{F}_{t_1}\r\}
+\mathbb{E}\{Q_{t_1,t_2}(u_{t_2}^x)^{-1}\nabla^{t_2}_2g(X_{t_1},X_{t_2})|\mathscr{F}_{t_1}\}
\end{align}
for $x\in M$ and $u_0\in \mathcal{O}_0(M)$. Fix the value of
$X_{t_1}^x=x_0$, by the assumption for $n=k$, we have
$$(u_{t_2}^x)^{-1}\nabla^{t_2}_2\mathbb{E}(f(x_0, X_{t_2}^x, \cdots, X_{t_{k+1}}^x)|\mathscr{F}_{t_2})=\sum_{i=2}^{k+1}\mathbb{E}\l\{Q_{t_2,t_i}^x(u_{t_i}^x)^{-1}\nabla_i^{t_i}f(x_0,X^x_{t_2},\cdots,X^x_{t_{k+1}})|\mathscr{F}_{t_2}\r\}.$$
Combining this with (\ref{s2-e2}), we have
\begin{align}\label{Der}
&(u_{t_1}^x)^{-1}\nabla^{t_1}_1\mathbb{E}\l\{f(X_{t_1}^x,X_{t_2}^x,\cdots,
X^x_{t_{k+1}})\bigg|\mathscr{F}_{t_1}\r\}=(u_{t_1}^x)^{-1}\nabla^{t_1}_1\mathbb{E}\{g(X^x_{t_1},
X^x_{t_2})|\mathscr{F}_{t_1}\}\nonumber\\
&=\sum_{i=1}^{k+1}\mathbb{E}\l\{Q^x_{t,t_i}(u^x_{t_i})^{-1}\nabla^{t_i}_if(X^x_{t_1},X^x_{t_2},\cdots,
X^x_{t_{k+1}})|\mathscr{F}_{t_1}\r\}.
\end{align}

\end{proof}
\begin{remark}\label{R1}
Especially, choosing $t_1=0$, we arrive at
\begin{align}\label{D1}(u_0^x)^{-1}\nabla^0\mathbb{E}\{f(X_{t_1}^x,X_{t_2}^x,\cdots,X_{t_n}^x)\}=\sum_{i=1}^n\mathbb{E}\l\{Q^x_{t_i}(u_{t_i}^x)^{-1}\nabla_i^{t_i}f(X_{t_1}^x,X_{t_2}^x,\cdots,X_{t_n}^x)\r\}.\end{align}
\end{remark}
The following result is a direct consequence of (\ref{Bismut}) and the It\^{o} formula for $f(X_t^x)$, i.e.
$$f(X_t^x)=f(x)+\sqrt{2}\int_0^t\l<(u_s^x)^{-1}\nabla^{s}P_{s,t}f(X_s^x),\vd B_s\r>.$$
\begin{lemma}\label{s2-l3}
For any $n\geq 1$, $0<t_1<t_2<\cdots< t_n\leq T$, and $f\in
C^{\infty}(M^n)$,
\begin{align*}
\mathbb{E} \l\{f(X_t^x)\int_0^t \l<h_s',\vd
B_s\r>\r\}=\mathbb{E}\int_0^t\l<(u_t^x)^{-1}\nabla^tf(X_t^x),(Q^x_{s,t})^*h_s\r>\vd
s, h\in \mathbf{\tilde{H}}_0, t\in [0,T]
\end{align*}
holds for all $x\in M$ and $u_0^x\in \mathcal{O}_0(M)$.
\end{lemma}

\section{ The Log-Sobolev Inequality}
When the metric is independent of $t$, log-Sobolov inequalities on $W_x^T$ were established independently by
Hsu \cite{Hsu02} and by Aida and  Elworthy \cite{AE}. In this section,
we first consider  the path space with a fixed initiated point,
then move to the free path space following an idea of \cite{FM}, where the
(non-damped) gradient operator is studied on the free path space
over the constant manifolds without boundary.

\subsection{ Log-Sobolev inequality on $W^T_x$}
 Let $\Pi^T_x$ be the distribution of $X^x_{[0,T]}$. Let
$$\mathscr{E}^x(F,G)=\mathbb{E}\l\{\l<D^0F, D^0G\r>_{\bf H_0}(X^x_{[0,T]})\r\},\  F,\ G\in \mathscr{F}C_0^{\infty}.$$
Since both $D^0F$ and $D^0G$ are functionals of $X$,
$(\mathscr{E}^x, \mathscr{F}C_0^{\infty})$ is a positive bilinear
form on $L^2(W^T_x; \Pi ^T_x)$. It is standard that the integration
by parts formula (\ref{In}) implies the closability of the form, see
Lemma \ref{s2-l4}. We shall use $(\mathscr{E}^x,
\mathscr{D}(\mathscr{E}^x))$ to denote the closure of
$(\mathscr{E}^x, \mathscr{F}C^{\infty}_0)$. Moreover, (\ref{In}) also
implies the Clark-Oc\^{o}ne type martingale representation formula, see
Lemma \ref{s2-l5}, which leads to the standard Gross log-Sobolev
inequality (see e.g.  \cite{Gross}).
\begin{lemma}\label{s2-l4}
Assume ${\bf(A)}$. $(\mathscr{E}^x, \mathscr{F}C_0^{\infty})$ is
closable in $L^2(W^T_x;\Pi^T_x)$.
\end{lemma}
\begin{proof}
By the integration by part formula,  the discussion is standard (see \cite[Lemma 4.1]{W11b}), we
omit it here.
\end{proof}
The following result gives us the Clark-Oc\^{o}ne  type martingale representation formula for $F(X_{[0,T]})$.
Following  the proof of \cite[Lemma 4.2]{W11b} for the case with constant metric, we have
\begin{lemma}\label{s2-l5}{\rm(Clark-Haussman-Oc\^{o}ne Formula)}
Assume ${\bf(A)}$. For any $F\in \mathscr{F}C_0^{\infty}$, let
$\tilde{D}^0F(X_{[0,T]}^x)$ be the projection of
$D^0F(X_{[0,T]}^x)$ on $\mathbf{\tilde{H}}_0$, i.e.
$$\frac{\vd}{\vd t}(\tilde{D}^0F(X_{[0,T]}^x))(t)=\mathbb{E}\l(\frac{\vd }{\vd t}(D^0F(X_{[0,T]}^x))\big|\mathscr{F}_t\r),\ t\in [0,T],\  (\tilde{D}^0F(X_{[0,T]}^x))(0)=0.$$
Then
$$F(X_{[0,T]}^x)=\mathbb{E}F(X^x_{[0,T]})+\int_0^T\l<\frac{\vd }{\vd t}(\tilde{D}^0F(X_{[0,T]}^x)(t),\vd B_t\r>.$$
\end{lemma}

It is standard that the martingale representation in Lemma
\ref{s2-l5} implies the following log-Sobolev inequality. Since the
parameter $T$ and the information of Ricci curvature and the second fundamental form  have been properly
contained in the Dirichlet form $\mathscr{E}$, the resulting log-Sobolev constant is
independent of $T$, $K$ and $\sigma$. Moreover,  it is well-known that the
constant 2 in the inequality is sharp  constant for compact manifolds with constant metric.
\begin{theorem}\label{s2-t2}
Assume $\bf{(A)}$. For any $T>0$ and $x\in M$,
$(\mathscr{E}^x,\mathscr{D}(\mathscr{E}^x))$ satisfies  the
following log-Sobolev inequality,
$$\Pi^T_x(F^2\log F^2)\leq 2\mathscr{E}^x(F,F), \  F\in \mathscr{D}(\mathscr{E}^x),\  \Pi_x^T(F^2)=1.$$
\end{theorem}
\begin{proof}
Due to lemma \ref{s2-l4}, it sufficient to prove the inequality for $F\in \mathscr{F}C_0^{\infty}$. Let
$$m_t:=\mathbb{E}\l(F(X^x_{[0,T]})^2\big|\mathscr{F}_t\r),\ \ t\in [0,T].$$
By the It\^{o} formula,
\begin{align*}
\vd m_t\log m_t=(1+\log m_t)\vd m_t+\frac{|\frac{\vd}{\vd t}(\tilde{D}^0F(X_{[0,T]}^x))(t)|^2}{2m_t}\vd t.
\end{align*}
Therefore,
\begin{align*}
\Pi_x^T(F^2\log F^2)&=\mathbb{E}_xm_T\log m_T=\int_0^T\frac{2\mathbb{E}\l(F(X_{[0,T]}^x)\frac{\vd}{\vd t}(D^0F(X_{[0,T]}^x)(t))\ \big|\mathscr{F}_t\r)^2}{\mathbb{E}\l(F(X_{[0,T]}^x)^2|\mathscr{F}_t\r)}\vd t\\
&\leq 2\int_0^T\mathbb{E}\l|\frac{\vd}{\vd t}(D^0F(X_{[0,T]}^x))(t)\r|^2\vd t\\
&=2\mathbb{E}\|D^0F(X_{[0,T]}^x)\|^2_{{\bf H}_0}\\
&=2\mathscr{E}^x(F,F).
\end{align*}
\end{proof}
Note that    on  manifolds without boundary equipped with time-depending metric, the log-Sobolov inequality with respect to the Dirichlet form induced by usual gradient derivative is recently established in \cite{Ch}. And the log-Sobolov constant is $2\exp\{\sup_{t\in [0,T]}|\mathcal{R}_t^{Z}|\}$.

\subsection{Application to  free path spaces}
Let $\Pi^T_{\mu}$ be the distribution of the (reflecting) diffusion
process generated by $L_t:= \nabla^t + Z_t$ with initial
distribution $\mu$ and time-interval $[0, T]$.
Due to the freedom of the initial point, it is natural for us to make use of the following Cameron-Martin space:
$${\bf H}=\l\{h\in C([0,T];\mathbb{R}^d): \int_0^T|h'(t)|^2\vd t<\infty\r\}.$$
 Then ${\bf H}$ is a
Hilbert space under the inner product
$$\l<h_1,h_2\r>_{\bf H}=\l<h_1(0), h_2(0)\r>+\int_0^T\l<h'_1(t),h'_2(t)\r>\vd t.$$

 To defined
the damped gradient operator on the free path space, let
$$\overline{\Omega}=M\times \Omega,\  \overline{\mathscr{F}}_t=\mathscr{B}(M)\times \mathscr{F}_t,\ \mbox{and} \ \overline{\mathbb{P}}=\mu\times \mathbb{P}.$$
 Let  $X_0(x,\omega)=x$ for $(x,\omega)\in M\times \Omega$. Then,
under the filtered probability space $(\overline{\Omega},\
\overline{\mathscr{F}}_t,\ \overline{\mathbb{P}})$, $X_t(x,\omega):=
X^x_t(\omega)$ is the (reflecting) diffusion process generated by
$L_t$  stating from $x$, and $u_t(x,\omega):=
u^x_t(\omega) $ is its horizontal lift. Moreover, let
$Q_{r,t}(x,\omega)=Q^x_{r,t}(\omega)$ for $0\leq r\leq t$, and write $Q_t^x(\omega):=Q_{0,t}(x,\omega)$ for simplicity. Now, for
any $F\in  FC^{\infty}_0$ with $F(\gamma) =
f(\gamma_{t_1},\cdots,\gamma_{t_n})$, let
\begin{align}\label{e5}DF(X)=D^0F(X) +\sum_{i=1}^n Q_{t_i}u_{t_i}^{-1}\nabla^{t_i}_if(X_{t_1},\cdots,X_{t_n}),\end{align}
where $ D^0F(X):= \sum _{i=1}^n\int_0^{t_i} Q_{t,t_i}
u_{t_i}^{-1}\nabla^{t_i}_i f(X_{t_1},X_{t_2},\cdots,X_{t_n})\vd t $ is the
damped gradient on the path space with fixed initial point.
Obviously, $ DF(X)\in  L^2(\overline{\Omega}\rightarrow {\bf
H};\overline{\mathbb{P}})$. Define the Dirichlet form by
$$\mathscr{E}^{\mu}(F,G)=\mathbb{E}_{\overline{\mathbb{P}}}\l<DF,DG\r>_{{\bf H}}, F,G\in \mathscr{F}C_0^{\infty}.$$
We aim to prove that $(\mathscr{E}^{\mu},FC^{\infty}_0)$ is closable
in $L^2(W^T; \Pi^T_{\mu})$ and then  establish the log-Sobolev
inequality for its closure $(\mathscr{E}^{\mu},
\mathscr{D}(\mathscr{E}^{\mu}))$. To prove the closability, we need
the following two lemmas modified from \cite{FM}. Let $\mathscr{H}_0(M)$
be the class of all smooth vector fields on $M $ with compact support.

Let div$^0_{\mu}$ be the divergence operator  w.r.t. $\mu$, which is  the
minus adjoint of $\nabla^0$ in $L^2(\mu)$; that is, for any
smooth vector field $U$, \begin{align}\label{e6}\int_M(Uf)\vd \mu=-\int_Mf({\rm
div}^0_{\mu}U)\vd \mu,\  f\in C^1_0(M).\end{align}

\begin{lemma}\label{s2-l7}
 Assume ${\bf(A)}$. For any $F\in \mathscr{F}C^{\infty}_0$, $U\in  \mathscr{H}_0(M)$,
and $\overline{\mathscr{F}}_t$-adapted $h\in L^2(\Omega\rightarrow
{\bf H};\overline{\mathbb{P}})$,
$$\mathbb{E}_{\overline{\mathbb{P}}}\l<DF(X),h+u_0^{-1}U(X_0)\r>_{{\bf H}}=\mathbb{E}_{\overline{\mathbb{P}}}\l\{F(X)\l(\int_0^T\l<h'(t),\vd B_t\r>-({\rm div}^0_{\mu}U)(X_0)\r)\r\}$$
\end{lemma}
One can mimic the proof of \cite{FW}, we omit it here. Due to this lemma,
 we have the following result, the main idea is standard.
\begin{theorem}\label{s2-t3}
Assume ${\bf(A)}$. $(\mathscr{E}^{\mu},\mathscr{F}C^{\infty}_0)$ is
closable in $L^2(W^T;\Pi^T_{\mu})$, and its closure is symmetric
Dirichlet form.
\end{theorem}
\begin{proof}
It suffices to prove the closability. Let $\{F_n\}_{g\geq 1}\subset\mathscr{F}C_0^{\infty}$ such that
$\lim\limits_{n\rightarrow \infty}F_n=0$ in $L^2(W^T;\Pi^T_{\mu})$ and $V:=\lim\limits_{n\rightarrow \infty}DF_n(X)$ exists in $L^2(\overline{\Omega}\rightarrow H; \overline{\mathbb{P}})$. We intend to prove
that $V=0$.
 Under the condition ${\bf(A)}$, by the decomposition of identity, there is a smooth ONB $\{U_i\}_{i=1}^d$ for the tangent space. Then for any $f\in C_0^{\infty}(M^n)$, define
 $$\xi_j=\sum_{i=1}^d\l<u_0Q_{t_i}u_{t_i}^{-1}\nabla_i^{t_i}f(X_{t_1}, X_{t_2},\cdots, X_{t_n}), U_j(X_0)\r>_0,\ \ j=1,2,\cdots,d,$$
 and there holds
\begin{align}\label{e4}\sum_{i=1}^{n} Q_{t_i} u_{t_i}^{-1}\nabla^{t_i}_if(X_{t_1}, X_{t_2},\cdots, X_{t_n})=\sum_{i=1}^{\infty}\xi_{j}u_0^{-1}U_j(X_0).\end{align}
Combining this with (\ref{e5}), it suffices to prove $\mathbb{E}_{\overline{\mathbb{P}}}\l<V,h+\xi u_0^{-1}U(X_0)\r>_{\bf H}=0$ for $\overline{\mathscr{F}}_t$-adapted $h\in L^2(\overline{\Omega}\rightarrow {\bf H}_0;\overline{\mathbb{P}}),\ \xi\in L^2(\overline{\Omega};\overline{\mathbb{P}})$ and $U\in\mathscr{H}_0(M)$.  Since $\mathscr{F}C_0^{\infty}$ is dense in $L^2(W^T;\Pi^T_{\mu})$, we may assume that $\xi=G(X)$ for some $G\in \mathscr{F}C_0^{\infty}$. In this case, it follows from Lemma \ref{s2-l7} and Eq. (\ref{e6})
that
\begin{align*}
&\mathbb{E}_{\overline{\mathbb{P}}}\l<V,h+\xi u_0^{-1} U(X_0)\r>_{\bf H}=\mathbb{E}_{\overline{\mathbb{P}}}\l<V(0),\xi u_0^{-1} U(X_0)\r>_{\bf H}=\lim_{n\rightarrow \infty}\mathbb{E}_{\overline{\mathbb{P}}}\l<DF_n(X), G(X)u_0^{-1}U(X_0)\r>_{\bf H}\\
&=\lim_{n\rightarrow \infty}\mathbb{E}_{\overline{\mathbb{P}}}\l\{\l<\{D(F_nG)(X)\}(0),u_0^{-1}U(X_0)\r>_{\bf H}-F_n(X)\l<((DG)(X))(0),u_0^{-1}U(X_0)\r>_{\bf H}\r\}\\
&=-\lim_{n\rightarrow \infty}\mathbb{E}_{\overline{\mathbb{P}}}\l\{F_n(X)\l(G(X)({\rm div}_{\mu}^0U)(X_0)+\l<((DG)(X))(0),u_0^{-1}U(X_0)\r>_{\bf H}
\r)\r\}\\
&=0.
\end{align*}

\end{proof}
By  Theorems \ref{s2-t2} and \ref{s2-t3}, we obtain the main result of this subsection as follows.
\begin{theorem}\label{s2-t4}
Assume ${\bf(A)}$. If the log-Sobolev inequality
\begin{align}\label{e1}\mu(f^2\log f^2)\leq C\mu(|\nabla^0f|_0^2),\ f\in C_b^1(M),\  \mu(f^2)=1\end{align}
holds for some constant $C>0$, then
$$\Pi^{T}_{\mu}(F^2\log F^2)\leq (2\vee C)\mathscr{E}^{\mu}(F,F),\  F\in \mathscr{D}(\mathscr{E}^{\mu}),\  \Pi^{T}_{\mu}(F^2)=1.$$
\end{theorem}
\begin{proof}
 From Theorem \ref{s2-t3}, it suffice to prove $F\in \mathscr{F}C_0^{\infty}$. By Theorem \ref{s2-t2} and the condition (\ref{e1}), we have
\begin{align}\label{e2}
\Pi^T_{\mu}(F^2\log F^2)&=\int_M\Pi_x^T(F^2\log F^2)\mu(\vd x)\nonumber\\
&\leq 2\int_M\mathscr{E}^x(F,F)\mu(\vd x)+C\int_M\Pi_x^T(F^2)\log \Pi_x^T(F^2)\mu(\vd x)\nonumber\\
&\leq 2\mathbb{E}_{\mathbb{P}}\|D^0F(X)\|^2_{{\bf H}_0}+C\int_M|\nabla^0\sqrt{\mathbb{E}^{\cdot}F^2(X)}|_0^2\vd \mu.
\end{align}
Moreover, letting $F(X)=f(X_{t_1},X_{t_2},\cdots,X_{t_n})$, it follows from Lemma \ref{s2-l2} that
\begin{align}\label{e3}
|\nabla^0\sqrt{\mathbb{E}^{\cdot}F^2(X)}|_0^2&=\frac{|\mathbb{E}^{\cdot}F(X)\sum_{i=1}^nQ_{t_i}u_{t_i}^{-1}\nabla^{t_i}_if(X_{t_1},\cdots,X_{t_n})|^2}{\mathbb{E}^{\cdot}F^2(X)}\nonumber\\
&\leq\mathbb{E}\big|\sum_{i=1}^nQ_{t_i}u_{t_i}^{-1}\nabla_i^{t_i}f(X_{t_1},\cdots, X_{t_n})\big|^2.\end{align}
Combining (\ref{e3}) with (\ref{e2}), we complete the proof.
\end{proof}
\section{Transportation-cost inequalities on path spaces over convex manifolds}

The main purpose of this section is to investigate the Talagrand type
inequalities on the path space $W^{S,T}:=C([S,T];M),\ 0\leq S<T<T_c$ of the (reflecting) diffusion
processes on the  manifold with convex boundary under $g_t$, $t\in [0,T_c)$. Let  $X_t$
be the (reflecting if $\partial M \neq \varnothing$) diffusion
process generated by $L_t$ with initial distribution $\mu\in
\mathscr{P}(M)$. Assume that $X_t$ is non-explosive, which is the case if
\begin{equation}\label{CV}
   \mathcal{R}_t^Z\geq K(t), \, \mbox{for\ some}\ K\in C([0,T_c)),\ \mathbb{I}_t\geq 0.
\end{equation}
We call the metric flow is convex flow if $(\partial M, g_t)$ keeps  convex, i.e. $\mathbb{I}_t\geq 0$.
 Let
$\Pi^{S,T}_{\mu}$ be the distribution of $X_{[S,T]}:=\{X_t: t\in [S,
T]\}$, $0\leq S<T<T_c$, which is a probability measure on the (free)
path space $W^{S,T}$. When $\mu=\delta _x$, we denote
$\Pi_{\delta_x}^{S,T}= \Pi_x^{S,T}$. For any nonnegative measurable
function $F$ on $W^{S,T}$ such that $\Pi^{S,T}_{\mu}(F)=1$, one has
\begin{align}\label{def2}\mu_F^{S,T}(\vd x):=\Pi^{S,T}_{x}(F)\mu (\vd x)\in \mathscr{P}(M).\end{align}

Consider the uniform distance on $W^{S,T}$:
$$\rho_{\infty}(\gamma, \eta):=\sup_{t\in [S,T]}\rho(\gamma_t,\eta_t),\gamma,\eta \in W^{S,T}.$$
Let $W^{\rho_{\infty}}_2$ be the $L^2$-Wasserstein distance (or
$L^2$-transportation cost) induced by $\rho_{\infty}$. In general,
for any $p\in [1,\infty)$ and for two probability measures
$\Pi_1,\Pi_2$ on $W^{S,T}$,
$$W^{\rho_{\infty}}_{p}(\Pi_1,\Pi_2):=\inf_{\pi \in \mathscr{C}(\Pi_1,\Pi_2)}\l\{\int_{W^{S,T}\times W^{S,T}}\rho_{\infty}(\gamma,\eta)^p\pi(\vd \gamma, \vd \eta)\r\}^{1/p}$$
is the $L^p$-Wasserstein distance (or $L^p$-transportation cost) of
$\Pi_1$ and $\Pi_2$, induced by the uniform norm, where
$\mathscr{C}(\Pi_1, \Pi_2)$ is the set of all couplings for $\Pi_1$
and $\Pi_2$. Similarly the $L^p$-Wasserstein distance of
$\mu$ and $\nu$ induced by $g_t$-distance defined by
$$W_{p,t}(\nu,\mu)=\inf_{\eta\in \mathscr{C}(\nu,\mu)}\l\{\int_{M\times M}\rho_t(x,y)^p\vd \eta(x,y)\r\}^{1/p}.$$

 The following
Theorem \ref{th1} provides that there are some  transportation-cost
inequalities to be equivalent to the lower bound of  $\mathcal{R}^{Z}_t$ and the convexity of $(\partial M,g_t)$, $t\in [0,T_c)$ (when $\partial M\neq
\varnothing$). To prove this result, we need the following
lemma due to \cite{OW}.

\begin{lemma}\label{LOV}
 Let $\mu$ be a probability measure on $M $  and $f\in C^2_b(M)$
such that $\mu(f)=0$. For small enough $\varepsilon >0$ such that
$f_{\varepsilon}:= 1+\varepsilon f\geq 0$, there holds
$$\mu(f^2)\leq \frac{1}{\varepsilon}\sqrt{\mu{(|\nabla^{s}f|_{s}^2)}}W_{2,s}(f_{\varepsilon}\mu, \mu)+\frac{\|{\rm Hess}^s_f\|_{\infty}}{2\varepsilon}W_{2,s}(f_{\varepsilon}\mu, \mu)^2,$$
 where $\|{\rm Hess}^s_f\|_{\infty}=\sup_{x\in M}\|{\rm Hess}^s_f\|_{op}$ for $\|\cdot\|_{op}$ the operator norm in $\mathbb{R}^d$.
\end{lemma}
The main result of the section is presented as follows.
\begin{theorem}\label{th1}
Let $P_{S,T}(x,\cdot)$  be the distribution of $X_T$ with conditional $X_S=x$.
Denote the corresponding inhomogeneous semigroup by $\{P_{S,T}\}_{0\leq S\leq T <T_c}$. For
any $p\in [1,\infty)$, the following
statements are equivalent to each other:
\begin{enumerate}
  \item [$(1)$]  \eqref{CV} holds.
  \item [$(2)$]For any $0\leq S\leq T< T_c$, $\mu \in  \mathscr{P}(M)$ and nonnegative $F $ with
  $\Pi^{S,T}_{\mu}(F)=1$,
  $$W_2^{\rho_{\infty}}(F\Pi^{S,T}_{\mu}, \Pi^{S,T}_{\mu^{S,T}_{F}})\leq 4C(S,T,K)\Pi^{S,T}_{\mu}(F\log F)$$
holds, where $\mu^{S,T}_F\in \mathscr{P}(M)$ is fixed by $(\ref{def2})$ and $C(S,T,K):=\sup\limits_{t\in [S,T]}\int_S^te^{-2\int_u^tK(r)\vd r}\vd u$, which will keep the same meaning in $(3),(7)$ and $(8)$.
  \item [$(3)$] For any $x\in M$ and  $0\leq S\leq T< T_c$,
  $$W^{\rho_{\infty}}_{2}(F\Pi_x^{S,T},\Pi_x^{S,T})^2\leq 4C(S,T,K) \Pi_x^{S,T}(F\log F),\ F\geq 0, \ \Pi_x^T(F)=1. $$
  \item [$(4)$] For any $x\in M$ and  $0\leq S\leq T< T_c$,
  $$W_{2,T}(P_{S,T}(x,\cdot), fP_{S,T}(x,\cdot))^2\leq 4\l( \int_S^Te^{-2\int_u^TK(r)\vd r}\vd u\r) P_{S,T}(f\log f)(x), f\geq 0,\ P_{S,T}f(x)=1.$$
  \item [$(5)$] For any $x\in M$ and  $0\leq S\leq T< T_c$,
  \begin{align*}W_{2,T}(P_{S,T}(x,\cdot), fP_{S,T}(x,\cdot))^2\leq 4\l(\int_S^Te^{-2\int_u^TK(r)\vd r}\vd u\r)^2P_{S,T}&\frac{|\nabla^Tf|_T^2}{f}(x),\end{align*}
  where $ f\geq 1$ and $P_{S,T}f=1$.
  \item [$(6)$]For $0\leq S\leq T< T_c$ and $\mu, \nu\in \mathscr{P}(M)$,
  $$W^{\rho_{\infty}}_p(\Pi^{S,T}_{\mu},\Pi_{\nu}^{S,T})\leq e^{-\int_S^TK(r)\vd r}W_{p,S}(\mu,\nu).$$
  \item [$(7)$] For any $0\leq S\leq T< T_c$, $\mu\in \mathscr{P}(M)$, and $F\geq
  0$ with $\Pi^{S,T}_{\mu}(F)=1$,
  $$W^{\rho_{\infty}}_2(F\Pi^{S,T}_{\mu}, \Pi^{S,T}_{\mu})\leq 2\l\{C(S,T,K)\Pi^{S,T}_{\mu}(F\log F)\r\}^{1/2}+e^{-\int_S^TK(r)\vd r}W_{2,S}(\mu_F^{S,T},\mu).$$
  \item [$(8)$] For any $\mu \in \mathscr{P}(M)$ and $C\geq 0$ such
  that
  $$W_{2,S}(f\mu,\mu)^2\leq C\mu(f\log f),\ f\geq 0,\  \mu(f)=1,$$
  there holds
  \begin{align*}W^{\rho_{\infty}}_2(F\Pi^{S,T}_{\mu},\Pi^{S,T}_{\mu})\leq \l(2\sqrt{C(S,T,K)}+\sqrt{C}e^{-\int_S^TK(r)\vd r}\r)^2\Pi^{S,T}_{\mu}(F\log
  F),\end{align*}
  for $ F\geq 0,\ \Pi^{S,T}_{\mu}(F)=1$.
\end{enumerate}
\end{theorem}
\begin{proof}

By taking $\mu=\delta_x$, we have $\mu_F^T=\Pi^T_{x}(F)\delta_x=\delta_x$. It is easy to see that
(3) follows from (2), (7) and (8). (4) follows from (3) by taking $F(X_{[S,T]})=f(X_T)$. (6) implies
$$W_{p,S}(\delta_xP_{S,T},\delta_yP_{S,T})\leq e^{-\int_S^TK(r)\vd r}\rho_{S}(x,y)$$
and thus implies (1) by \cite[Theorem 5.3]{Cheng}. Moreover, it is clear that (8) follows from (7) while
(7) is implied by each of (2) and (6). ``(3)$\Rightarrow$(2)" is the same as explained in time-homogeneous case (see \cite[the proof of Theorem 1.1 (b)]{W09}). So it suffices to prove (1) $\Rightarrow$ (3), each of (4) and (5) $\Rightarrow$ (1), (1) $\Rightarrow$ (5), (1) $\Rightarrow$ (6). Without loss generality, we assume
$S=0$ for simplicity.

${(a)}\  \mathbf{(1)\ implies\  (3)}$  We shall only consider the case where $\partial M$ is non-empty. For the case without boundary, the following argument works well by taking $l_t=0$ and $N_t=0$. Simply denote $X^x_{[0,T]}=X_{[0,T]}$. Let $F$ be a positive bounded measurable function on $W^T$ such that $\inf F>0$
and $\Pi^T_{x}(F)=1$. Let $\vd \mathbb{Q}=F(X_{[0,T]})\vd \mathbb{P}$. Since $\mathbb{E}F(X_{[0,T]})=\Pi^T_{\mu}(F)=1$, $\mathbb{Q}$ is a probability measure on $\Omega$. Then, with a similar discussion as in \cite{W11b}, we conclude that there
exists a unique $\mathscr{F}_t$-predict process $\beta_t$ on $\mathbb{R}^d$  such that
$$F(X_{[0,T]})=m_T=e^{\int_0^T\l<\beta_s, \vd B_s\r>-\frac{1}{2}\int_0^T\|\beta_s\|^2\vd s}$$
and
\begin{align}\label{eq13}
\int_0^T\mathbb{E}_{\mathbb{Q}}\|\beta_s\|^2\vd s=2\mathbb{E}F(X_{[0,T]}\log X_{[0,T]}).
\end{align}
Then by the Girsanov theorem, $\tilde{B}_t:=B_t-\int_0^t\beta_s\vd s,\ \ t\in [0,T]$
is a $d$-dimensional Brownian motion under the probability measure $\mathbb{Q}$.

Let $Y_t$ solve the following SDE
\begin{align}\label{eq7}\vd Y_t=\sqrt{2}P^t_{X_t,Y_t}u_t\circ\vd \tilde{B}_t+Z_t(Y_t)\vd t +N_t(Y_t)\vd \tilde{l}_t,\ \ Y_0=x,\end{align}
where $P^t_{X_t,Y_t}$ is the $g_t$-parallel displacement along the minimal geodesic from $X_t$ to $Y_t$ and $\tilde{l}_t$ is the local time of $Y_t$ on $\partial M$. As announced, under $\mathbb{Q}$, $\tilde{B}_t$ is a $d$-dimensional Brownian motion, the distribution of $Y_{[0,T]}$ is $\Pi_{x}^T$.

On the other hand, since $\tilde{B}_t=B_t-\int_0^t\beta_s\vd s$, we have
\begin{align}\label{eq8}\vd X_t=\sqrt{2}u_t\circ \vd \tilde{B}_t+Z_t(X_t)\vd t+\sqrt{2}u_t\beta_t\vd t+N_t(X_t)\vd l_t. \end{align}
Moreover, for any bounded measurable function $G$ on $W^T$,
$$\mathbb{E}_{\mathbb{Q}}G(X_{[0,T]}):=\mathbb{E}(FG)(X_{[0,T]})=\Pi_x^T(FG).$$
We conclude that the distribution of $X_{[0,T]}$ under $\mathbb{Q}$ coincides  with $F\Pi^T_x$. Therefore,
\begin{align}\label{eq6}
W^{\rho_{\infty}}_2(F\Pi_x^T,\Pi_x^T)^2&\leq \mathbb{E}_{\mathbb{Q}}\rho_{\infty}(X_{[0,T]},Y_{[0,T]})^2
=\mathbb{E}_{\mathbb{Q}}\max_{t\in [0,T]}\rho_t(X_t,Y_t)^2.
\end{align}
By the convexity of $(\partial M,g_t)$, we have
$$\l<N_t(x),\nabla^t\rho_t(\cdot,y)(x)\r>_t=\l<N_t(x),\nabla^t\rho_t(y,\cdot)(x)\r>_t\leq 0.$$
Combining this with (\ref{eq7}) and (\ref{eq8}), and by the It\^{o} formula, we obtain from $\mathcal{R}_t^Z\geq K(t)$ that
\begin{align}\label{eq9}
\vd \rho_t(X_t,Y_t)\leq &-K(t)\rho_t(X_t,Y_t)\vd t+\sqrt{2}\l<u_t\beta_t,\nabla^t\rho_t(\cdot,Y_t)(X_t)\r>_t\vd t\nonumber\\
\leq & (-K(t)\rho_t(X_t,Y_t)+\sqrt{2}\|\beta_t\|)\vd t.
\end{align}
Since $X_0=Y_0=x$, this implies
\begin{align}\label{eq10}
\rho_t(X_t,Y_t)^2&\leq e^{-2\int_0^tK(r)\vd r}\l(\sqrt{2}\int_0^te^{\int_0^sK(r)\vd r}\|\beta_s\|\vd s\r)^2\nonumber\\
&\leq 2e^{-2\int_0^tK(r)\vd r}\int_0^te^{2\int_0^sK(r)\vd r}\vd s\cdot \int_0^t\|\beta_s\|^2\vd s,\ \ t\in [0,T].
\end{align}
Therefore,
\begin{align}\label{eq11}
\mathbb{E}_{\mathbb{Q}}\max_{t\in [0,T]}\rho_t(X_t,Y_t)^2
\leq 2\max_{t\in [0,T]}\int_0^te^{-2\int_s^tK(r)\vd r}\vd s\int_0^t\mathbb{E}_{\mathbb{Q}}\|\beta_s\|^2\vd s.
\end{align}
Therefore, (3) follows from (\ref{eq6}) and (\ref{eq13}).

${(b)}\  \mathbf{(4)\ implies\  (1) \ } $\  Let $f\in C_b^2(M)$ such that $P_{0,T}f(x)=0$. Then, for small
$\varepsilon>0$ such that $f_{\varepsilon}:=1+\varepsilon f\geq 0$, we have
\begin{align*}
P_{0,T}(f_{\varepsilon}\log f_{\varepsilon})=P_{0,T}\l\{(1+\varepsilon f)\l(\varepsilon f-\frac{1}{2}(\varepsilon f)^2+o(\varepsilon ^2)\r)\r\}(x)=\frac{\varepsilon ^2}{2}P_{0,T}f^2(x)+o(\varepsilon ^2).
\end{align*}
Combining  with lemma \ref{LOV} and (4), we obtain
\begin{align*}
(P_{0,T}f^2)^2(x)&\leq 4\int_0^Te^{-2\int_u^TK(r)\vd r}\vd u\cdot P_{0,T}|\nabla^Tf|^2_T(x)\cdot\lim_{\varepsilon\rightarrow 0}\frac{P_{0,T}f_{\varepsilon}\log f_{\varepsilon}(x)}{\varepsilon ^2}\\
& 4\int_0^Te^{-2\int_u^TK(r)\vd r}\vd u\cdot P_{0,T}|\nabla^Tf|^2_T(x)P_{0,T}f^2(x).
\end{align*}
This is equivalent to \cite[Theorem 5.3]{Cheng} for $\sigma=0, p=2$ and continuous function $K$. Therefore,by \cite[Theorem 5.3]{Cheng} ``(2)$\Leftrightarrow$ (1)'',  (4) implies (1).

${(c)} \mathbf{(5)\ is \ equivalent \ to\ (1)}$. Similarly to $\mathbf{(b)}$, combining condition (5) with Lemma \ref{LOV}, we obtain
\begin{align*}
P_{0,T}f^2(x)&\leq 2\int_0^Te^{-2\int_s^TK(r)\vd r}\vd s\sqrt{P_{0,T}|\nabla^Tf|^2_T(x)}\lim_{\varepsilon\rightarrow 0}\sqrt{P_{0,T}\frac{|\nabla^T f_{\varepsilon}|^2_T}{f_{\varepsilon}\varepsilon^2}(x)}\\
&=2\int_0^Te^{-2\int_s^TK(r)\vd r}\vd sP_{0,T}|\nabla^Tf|^2_T.
\end{align*}
Hence, (1) holds.

 On the other hand, due to $(1) \Rightarrow (4)$ and  \eqref{CV},
\begin{align*}
P_{0,T}(f\log f)(x)\leq \int_0^Te^{-2\int_s^TK(r)\vd r}\vd s\cdot P_{0,T}\frac{|\nabla^Tf|_{T}^2}{f}(x),\ \ f\geq 0,\ \ P_{0,T}f(x)=1,
\end{align*}
we conclude that (1) implies (5).

${(e)} \mathbf{(1)\  implies\   (2)}$. For any $x,y\in M$, there exists $\Pi_{x,y}\in \mathscr{C}(\Pi_x^T,\Pi_y^T)$ such that
$$\int_{W^T\times W^T}\rho_{\infty}^p\vd \Pi_{x,y}\leq e^{-p\int_0^TK(r)\vd r}\rho_0(x,y)^p.$$
The following discussion is similar with the constant matric case (see \cite[Theorem 1.1]{W09}).
\end{proof}

\section{Extension to  non-convex setting}
In this section, we first consider
$L_t=\psi^{2}_t(\nabla^t+Z_t)$ with diffusion coefficient $\psi_t$ on
manifolds with convex  flow; then extend these results to non-convex case.

\subsection{The case with a diffusion coefficient}
\quad Let $\psi_t(\cdot)=\psi(t, \cdot)>0$ be a smooth function on $(M,g_t)$ and
 $\Pi^T_{\mu, \psi}$ be the distribution of the (reflecting if
$\partial M\neq \varnothing$) diffusion process generated by
$L_t=\psi_t^2(\nabla^t+Z_t)$ on time interval $[0,T]\subset [0,T_c)$
with initial distribution $\mu$. Set
$\Pi^T_{x,\psi}=\Pi^T_{\delta_x, \psi}$ for $x\in M$. Moreover, for
$F\geq 0$ with $\Pi^T_{\mu, \psi}=0$, let $\mu ^T_{F,\psi}(\vd
x)=\Pi^T_{s,\psi}(F)\mu(\vd x).$
\begin{theorem}\label{s3-t2}
 Assume that  $\mathbb{I}_t\geq 0$ for $t\in [0,T_c)$ and  $\mathrm{Ric}_t^Z\geq K_1(t),\ \mathcal{G}_t\leq K_2(t)$ for some continuous function on $[0,T_c)$. Let $\psi \in C^{1,\infty}_b([0,T_c)\times M)$ be strictly positive.
Let
$$K_{\psi}(t)=(d-1)\|\nabla^t\psi_t\|_{\infty}^2+K_1^-(t)\|\psi_t\|_{\infty}^2+2\|Z_t\|_{\infty}\|\psi_t\|_{\infty}\|\nabla^t\psi_t\|_{\infty}+K_2(t).$$
Then
$$W^{\rho_{\infty}}_2(F\Pi^T_{\mu, \psi},\Pi^T_{\mu^{T}_{F,\psi}, \psi})^2\leq C(T,\psi)\Pi^{T}_{\mu,\psi}(F\log F),\ \mu\in \mathscr{P}(M),\ F\geq 0, \ \Pi^{T}_{\mu, \psi}(F)=1$$
holds for
$$C(T,\psi):=\inf_{R>0}\l\{4(1+R^{-1})\int_0^T\|\psi_s\|_{\infty}^2e^{2\int_s^TK_{\psi}(r)\vd r}\vd s
\cdot \exp\bigg[8(1+R)\sup_{s\in [0,T]}\|\nabla^{s}\psi_s\|_{\infty}\bigg]\r\}.$$
\end{theorem}
\begin{proof}
We shall only consider the case that $\partial M$ is non-empty. As explained in the proof of ``(3)$\Rightarrow $(2)" in \cite[Theorem 4.1]{W09}, it suffices to prove for $\mu=\delta_{x}$, $x\in M$.
In this case, the desired inequality reduces to
$$W_2^{\rho_{\infty}}(F\Pi_{x,\psi}^T,\Pi_{x,\psi}^T)\leq 2 C(T,\psi)\Pi _{x,\psi}^T(F\log F),\ \ F\geq 0,\ \Pi_{x,\psi}^T(F)=1.$$
Since the diffusion coefficient is non-constant, it is convenient  to adopt the It\^{o} differential $\vd_{I}$ for the Girsanov transformation. So the $L_t$-reflecting diffusion process can be constructed by solving the It\^{o} SDE
$$\vd _I X_t=\sqrt{2}\psi_t(X_t)u_t\vd B_t+\psi_t^2(X_t)Z_t(X_t)\vd t+N_t(X_t)\vd l_t,\ \ \ X_0=x,$$
where $B_t$ is the $d$-dimensional Brownian motion with natural filtration $\mathscr{F}_t$.
Let $\beta_t$, $\mathbb{Q}$ and $\tilde{B}_t$ be the same as in the proof of Theorem \ref{th1}. Then
\begin{align}\label{eq1}\vd _I X_t=\sqrt{2}\psi_t(X_t)u_t\vd \tilde{B}_t+\{\psi_t^2(X_t)Z_t(X_t)+\sqrt{2}\psi_t(X_t)u_t\beta_t\}\vd t+N_t(X_t)\vd l_t.\end{align}
Let $Y_t$ solve
\begin{align}\label{eq2}\vd_I Y_t=\sqrt{2}\psi_t(Y_t)P^t_{X_t,Y_t}u_t\vd \tilde{B}_t+\psi_t^2(Y_t)Z_t(Y_t)\vd t+N_t(Y_t)\vd \tilde{l}_t,\ \ Y_0=y, \end{align}
where
$\tilde{l}_t$ is the local time of $Y_t$ on $\partial M$. As explained in the above theorem (see e.g. the proof of Theorem \ref{th1} (a)), under $\mathbb{Q}$, the distribution of $Y_{[0,T]}$ and $X_{[0,T]}$ are $\Pi_{x,\psi}^T$ and
$F\Pi_{x,\psi}^T$, so
\begin{align}\label{eq3}
W_2^{\rho_{\infty}}(F\Pi_{x,\psi}^T,\Pi_{x,\psi}^T)\leq \mathbb{E}_{\mathbb{Q}}\max_{t\in [0,T]}\rho_t(X_t,Y_t)^2.
\end{align}
Noting that due to the convexity of the boundary,
$$\l<N_t(x),\nabla^t\rho(\cdot,y)(x)\r>_t=\l<N_t(y),\nabla^t\rho_t(y,\cdot)(x)\r>_t\leq 0,\ \ x\in \partial M,$$
Combining this with  (\ref{eq1}), (\ref{eq2})  and the comparison theorem \cite[Theorem 4.1]{Cheng}, we obtain
\begin{align*}
\vd \rho_t(X_t,Y_t)\leq& \sqrt{2}(\psi_t(X_t)-\psi_t(Y_t))\l<\nabla^t\rho_t(\cdot,Y_t)(X_t), u_t\vd B_t\r>_t\\
&+K_{\psi}(t)\rho_t(X_t,Y_t)\vd t+\sqrt{2}\|\psi_t\|_{\infty}\|\beta_t\|\vd t,
\end{align*}
where $$K_{\psi}(t)=(d-1)\|\nabla^t\psi_t\|_{\infty}^2+K^-_1(t)\|\psi_t\|^2_{\infty}+2\|Z_t\|_{\infty}\|\nabla^t\psi_t\|_{\infty}\|\psi_t\|_{\infty}+K_2(t).$$
Then
$$M_t:=\sqrt{2}\int_0^te^{-\int_0^sK_{\psi}(r)\vd r}(\psi_s(X_s)-\psi_s(Y_s))\l<\nabla^{s}\rho_{s}(\cdot, Y_s)(X_s),u_s\vd \tilde{B}_s\r>_{s}$$
is a $\mathbb{Q}$-martingale such that
\begin{align*}
\rho_t(X_t,Y_t)\leq e^{\int_0^tK_{\psi}(r)\vd r}\l(M_t+\sqrt{2}\int_0^te^{-\int_0^sK_{\psi}(r)\vd r}\|\psi_s\|_{\infty}\|\beta_s\|\vd s\r),\ \ t\in [0,T].
\end{align*}
So by the Doob inequality, we obtain
\begin{align*}
h_t:=&e^{-2\int_0^tK_{\psi}(s)\vd s}\mathbb{E}\max_{s\in [0,t]}\rho_{s}(X_s,Y_s)^2\\
\leq & (1+R)\mathbb{E}_{\mathbb{Q}}\max_{s\in [0,t]}M_s^2+2(1+R^{-1})\mathbb{E}_{\mathbb{Q}}\l(\int_0^te^{-\int_0^sK_{\psi}(r)\vd r}\|\psi_s\|_{\infty}\|\beta_s\|\vd s\r)^2\\
\leq &4(1+R)\mathbb{E}_{\mathbb{Q}}M_t^2+2(1+R^{-1})\int_0^te^{-2\int_0^sK_{\psi}(r)\vd r}\|\psi_s\|_{\infty}^2\vd s\int_0^t\mathbb{E}_{\mathbb{Q}}\|\beta_s\|^2\vd s\\
\leq & 8(1+R)\sup_{s\in [0,T]}\|\nabla^{s}\psi_s\|_{\infty}\int_0^th_s\vd s\\
&
+2(1+R^{-1})\int_0^T\|\psi_s\|_{\infty}^2e^{-2\int_0^sK_{\psi}(r)\vd r}\vd s\cdot\int_0^T\mathbb{E}\|\beta_s\|^2\vd s,\ \ \ t\in [0,T].
\end{align*}
 Since $h_0=0$, this inequality implies
\begin{align*}
&e^{-2\int_0^TK_{\psi}(s)\vd s}\mathbb{E}_{\mathbb{Q}}\max_{s\in [0,T]}\rho_{s}(X_s,Y_s)^2=h_T\\
&\leq 2(1+R^{-1})\int_0^T\|\psi_s\|_{\infty}^2e^{-2\int_0^sK_{\psi}(r)\vd r}\vd s\cdot \exp \l[8(1+R)\sup_{s\in[0,T]}\|\nabla^{s}\psi_s\|_{\infty}\r]\cdot \int_0^T\mathbb{E}_{\mathbb{Q}}\|\beta_s\|^2\vd s.
\end{align*}
As explained in (\ref{eq13}), it holds $$\int_0^T\mathbb{E}_{\mathbb{Q}}\|\beta_s\|^2\vd s=2\mathbb{E}F(X_{[0,T]})\log F(X_{[0,T]}).$$
Therefore
\begin{align*}
&\mathbb{E}\max_{s\in [0,T]}\rho_{s}(X_s,Y_s)^2\\
&\leq 4(1+R^{-1})\int_0^T\|\psi_s\|_{\infty}^2e^{2\int_s^TK_{\psi}(r)\vd r}\vd s
\cdot \exp{\big[8(1+R)\sup_{s\in [0,T]}\|\nabla^{s}\psi_s\|_{\infty}\big]\Pi_{o,\psi}^T(F\log F)}.
\end{align*}
Combining with (\ref{eq3}), we complete the proof.
\end{proof}
\begin{theorem}\label{s3-t1}
In the situation of Theorem \ref{s3-t2},
$$W^{\rho_{\infty}}_2(\Pi^T_{\nu,\psi},\Pi^T_{\mu,\psi})\leq 2e^{\int_0^T(K_{\psi}(t)+\|\nabla^t\psi_t\|_{\infty})\vd t}W_{2,0}(\nu,\mu).$$
\end{theorem}
\begin{proof}
As explained in the proof of Theorem \ref{th1} $``(6)\Rightarrow(5)"$, we only consider $\nu=\delta_x$, and $\nu=\delta_y$. Let $X_t$ and $Y_t$ solve the following SDEs respectively.
\begin{align*}
&\vd_I X_t=\sqrt{2}\psi_t(X_t)u_t\vd B_t+\psi_t^2(X_t)Z_t(X_t)\vd t+N_t(X_t)\vd l_t,\ \ X_0=x;\\
&\vd _I Y_t=\sqrt{2}\psi_t(Y_t)P_{X_t,Y_t}^tu_t\vd B_t +\psi_t^2(Y_t)Z_t(Y_t)\vd t+N_t( Y_t)\vd \tilde{l_t}, \ \ Y_0=y.
\end{align*}
Then, as explained in Theorem \ref{s3-t2}, by the It\^{o} formula,
\begin{align}
\vd \rho_t(X_t,Y_t)&\leq \sqrt{2}(\psi_t(X_t)-\psi_t(Y_t))\l<\nabla^t\rho_t(\cdot,Y_t)(X_t),u_t\vd B_t\r>_t+K_{\psi}(t)\rho_t(X_t,Y_t)\vd t.
\end{align}
Therefore, \begin{align}\label{Ieq2}\rho_t(X_t,Y_t)\leq e^{\int_0^tK_{\psi}(s)\vd s}(M_t+\rho_0(x,y)),\ t\geq 0,\end{align}
for $M_t:=\sqrt{2}\int_0^te^{-\int_0^sK_{\psi}(u)\vd u}(\psi_s(X_s)-\psi_s(Y_s))\l<\nabla^{s}\rho_{s}(\cdot, Y_s)(X_s),u_s\vd B_s\r>_{s}.$ Again
using the It\^{o} formula,
\begin{align*}
\vd \rho_t^2(X_t,Y_t)\leq \vd \tilde{M}_t+2\l[K_{\psi}(t)+\|\nabla^t\psi_t\|_{\infty}^2\r]\rho_t(X_t,Y_t)^2\vd t,
\end{align*}
where $\vd \tilde{M}_t=2\rho_t(X_t,Y_t)(\psi_t(X_t)-\psi_t(Y_t))\l<\nabla^t\rho_t(\cdot,Y_t)(X_t),u_t\vd B_t\r>.$ This implies
\begin{align*}
\vd \rho_t^2(X_t,Y_t)\leq e^{2\int_0^t(K_{\psi}(s)+\|\nabla^{s}\psi_s\|_{\infty})\vd s}\rho_0(x,y)^2.
\end{align*}
Combining this with (\ref{Ieq2}),  we arrive at
\begin{align*}
W_2^{\rho_{\infty}}(\Pi^T_{x,\psi},\Pi^T_{y,\psi})^2&\leq \mathbb{E}\max_{t\in [0,T]}\rho_t(X_t,Y_t)^2\leq e^{2\int_0^TK_{\psi}(t)\vd t}\mathbb{E}\max_{t\in [0,T]}(M_t+\rho_0(x,y))^2\\
&\leq 4e^{2\int_0^TK_{\psi}(t)\vd t}\mathbb{E}(M_T+\rho_0(x,y))^2=4e^{2\int_0^TK_{\psi}(t)\vd t}\mathbb{E}(M_T^2+\rho_0^2(x,y))\\
&\leq 4e^{2\int_0^TK_{\psi}(t)\vd t}\l(\rho_0(x,y)^2+2\int_0^Te^{-2\int_0^tK_{\psi}(s)\vd s}\|\nabla^t\psi_t\|_{\infty}\mathbb{E}\rho_t(X_t,Y_t)^2\vd t\r)\\
&\leq 4e^{2\int_0^T(K_{\psi}(t)+\|\nabla^t\psi_t\|_{\infty})\vd t}\rho_0(x,y)^2
\end{align*}
where the second inequality is due to the Doob inequality. This implies the desired inequality for
$\mu=\delta_x$ and $\nu=\delta_y$.
\end{proof}
\subsection{Non-convex manifold}
As discussed in Theorem \ref{s3-t2} and  with a proper conformal change of metric, we are able to
establish the following transportation-cost inequality on a class of manifolds with non-convex boundary.
Let
  $$\mathscr{D}=\{\phi\in C_b^2([0,T_c)\times M): \inf \phi_t=1,  \ \mathbb{I}_t\geq -N_t\log \phi_t\}.$$
Assume that $\mathscr{D}\neq \varnothing$ and for some $K_1, K_2\in C{([0,T_c))}$ such that
\begin{align}\label{3n1}{\rm Ric}_t^{Z}\geq K_1(t), \ \ \mathcal{G}_t\leq K_2(t)\end{align}
holds. To make the boundary convex, let $\phi_t\in \mathscr{D}$. By Theorem \cite[Lemma\ 2.1]{W07},
$\partial M$ become convex under $\tilde{g}_t=\phi_t^{-2}g_t$.
Let $\tilde{\Delta}_t$ and $\tilde{\nabla}^t$ be the Laplacian and gradient induced
by the new metric $\tilde{g}_t$.  Since $\phi_t\geq 1$, $\rho_t(x,y)$ is large than
$\tilde{\rho}_t(x,y)$, the Riemannian $\tilde{g}_t$-distance between $x$ and $y$.

\begin{theorem}\label{th2}
Let $\partial M\neq \varnothing$ and $\mathbb{I}_t\geq -\sigma(t)$ for positive $\sigma\in C([0,T_c))$. Assume \eqref{3n1} holds.  For $\phi\in\mathscr{D}$, let
$$K_{\phi}(t):=(d-1)\|\nabla^t\phi_t\|_{\infty}^2+K_{\phi,1}^-(t)+2\|\phi_tZ_t+(d-2)\nabla^t\phi_t\|_{\infty}\|\nabla^t\phi_t\|_{\infty}+K_{\phi,2}(t),$$
where
\begin{align*}
&K_{\phi,1}(t):=\inf\{\phi_tK_1(t)+\frac{1}{2}L_t\phi_t^2-|\nabla^t\phi_t^2|_t|Z_t|_t-(d-2)|\nabla^t\phi_t|_t^2\},\\
&K_{\phi,2}(t):=\sup\{-2\partial_t\log \phi_t+K_2(t)\}.
\end{align*}
Then for any $\mu\in \mathscr{P}(M)$,
$$W^{\rho_{\infty}}_2(F\Pi_{\mu}^T,\Pi_{\mu_F^T}^T)\leq \sup_{s\in [0,T]}\|\phi_t\|_{\infty}^2C(T,\phi)\Pi_{\mu}^T(F\log F),\ \ F\geq 0,\ \Pi_{\mu}^T(F)=1$$
holds for
$$C(T,\phi)=\inf_{R>0}\l\{4(1+R^{-1})\int_0^Te^{2\int_s^TK_{\phi}(r)\vd r}\vd s\exp{\bigg[8(1+R)\sup_{s\in [0,T]}\|\nabla^{s}\phi_s\|_{\infty}\bigg]}\r\}.$$
\end{theorem}
\begin{proof}
According to the proof of \cite[Proposition 4.7]{Cheng}. We have $L_t:=\phi_t^{-2}(\tilde{\Delta}_t+\tilde{Z}_t)$, where $\tilde{Z}_t=\phi_t^2Z_t+\frac{d-2}{2}\nabla^t\phi_t^2$, and $$\widetilde{{\rm Ric}}^{\tilde{Z}}_t\geq K_{\phi,1}(t),\ \ \ \tilde{\mathcal{G}}_t\leq K_{\phi,2}(t).$$
Let $K_{\psi}$ be defined in Theorem \ref{th2} for  the manifold equipped with $\tilde{g}_t$. Then, $L_t=\psi_t^{2}(\tilde{\Delta}_t+\tilde{Z}_t)$, where $\psi_t=\phi_t^{-1}$, we see that
$K_{\psi}(t)\leq K_{\phi}(t)$ and thus $C(t,\psi)\leq C(t,\phi)$. Hence, it follows from Theorem \ref{th2} that
$$W^{\tilde{\rho}_{\infty}}_2(F\Pi_{\mu}^T,\Pi^T_{\mu^T_F})^2\leq C(T,\phi)\Pi_{\mu}^T(F\log F),\ \ F\geq 0,\ \ \Pi_{\mu}^T(F)=1,$$
where $\tilde{\rho}_{\infty}$ is the uniform distance on $W^T$ induced by the metric $\tilde{g}_t$.
The proof is completed by $\rho_{\infty}\leq \sup_{t\in [0,T]}\|\phi_t\|_{\infty}\tilde{\rho}_{\infty}.$
\end{proof}

Since $K_{\psi}(t)\leq K_{\phi}(t)$ and
$\tilde{\rho}_t\leq \rho_t\leq \|\phi_t\|_{\infty}\tilde{\rho}_t$,
the following result follows from the proof of Theorem \ref{s2-t1} by taking $\psi=\phi^{-1}$.
\begin{theorem}\label{th3}
In the situation of Theorem \ref{th2},
$$W^{\rho_{\infty}}_{2}(\Pi^T_{\mu},\Pi_{\nu}^T)\leq 2\sup_{t\in [0,T]}\|\phi_t\|_{\infty}e^{\int_0^T(K_{\phi}(t)+\|\nabla^t\phi_t\|_{\infty})\vd t}W_{2,0}(\nu,\mu),\ \ \mu,\nu\in \mathscr{P}, T>0.$$
\end{theorem}

\bigskip
\bigskip

\noindent\textbf{Acknowledgements}  \ The author would thank Professor Feng-Yu Wang for valuable suggestions and  this work is supported in part
by 985 Project,
 973 Project.

\end{document}